\documentclass[12pt,a4paper]{amsart}
\usepackage{tikz,xcolor}
\newtheorem{theo+}              {Theorem}           [section]
\newtheorem{prop+}  [theo+]     {Proposition}
\newtheorem{coro+}  [theo+]     {Corollary}
\newtheorem{lemm+}  [theo+]     {Lemma}
\newtheorem{exam+}  [theo+]     {Example}
\newtheorem{rema+}  [theo+]     {Remark}
\newtheorem{defi+}  [theo+]     {Definition}
\def \r{\mbox{${\mathbb R}$}}

\newenvironment{theorem}{\begin{theo+}}{\end{theo+}}

\newenvironment{definition}{\begin{defi+}}{\end{defi+}}

\usepackage{amsthm}
\theoremstyle{plain} \theoremstyle{remark}
\newtheorem{remark}{Remark}
\newtheorem{example}{Example}

\newtheorem*{ack}{\bf Acknowledgments}

\def\E{/\kern-1.0em \equiv }

\title{Bi-eigenmaps and biharmonic submanifolds in a sphere}
\author{Ye-Lin Ou$^{*}$ }

\address{Department of Mathematics,\newline\indent
Texas A $\&$ M University-Commerce,\newline\indent Commerce, TX
75429, U S A.\newline\indent E-mail:yelin.ou@tamuc.edu}
\thanks{$^{*}$ This work was supported by a grant from the Simons Foundation ($\#427231$, Ye-Lin Ou).}
\begin{document}

\title[Bi-eigenmaps and biharmonic submanifolds]{Bi-eigenmaps and biharmonic submanifolds in a sphere}

\date{01/09/2022}
\maketitle

\maketitle
\section*{Abstract}
\begin{quote}
{\footnotesize In this note, we first give a classification of biharmonic submanifolds in a sphere defined by  bi-eigenmaps ($\Delta^2 \phi=\lambda \phi$) or  buckling eigenmaps ($\Delta^2 \phi=-\mu \Delta \phi$). We then classify  biharmonic bi-eigenmaps and buckling eigenmaps into spheres with constant energy density.  The results can be viewed as generalizations of Takahashi's characterization of minimal submanifolds in a sphere by eigenmaps. \\
}
\end{quote}

\section{Biharmonic maps and biharmonic subamnifolds}

Biharmonic maps are maps between Riemannian manifolds $\phi:(M, g)\to (N, h)$ which are critical points of the bienergy functional
\begin{align}
E_2(\phi)=\frac{1}{2}\int_M |\tau(\phi)|^2\,dv_g,
\end{align}
where $\tau(\phi)={\rm Tr}_g\nabla d\phi$ is the tension field of the map $\phi$, vanishing of which means the map is a harmonic map.
Biharmonic maps are generalizations of  biharmonic functions and harmonic maps. The latter class of maps include harmonic functions, geodesics, minimal submanifolds, and Riemannian submersions with minimal fibers as important special cases. 

Biharmonic submanifolds are the images of isometric immersions which are biharmonic maps. Biharmonic submanifolds are a generalization of minimal submanifolds as the latter are the images of harmonic isometric immersions.  There are many biharmonic submanifolds which are not minimal, and it is a custom to call them proper biharmonic submanifolds.

The study of biharmonic submanifolds has been attracting a growing interest in the past two decades with many interesting results and progress made in  the topics related to the following conjectures about the classifications of biharmonic submanifolds in a space form.

{\em 1. B. -Y. Chen's Conjecture \cite{Ch}}: There is no proper biharmonic submanifold in a Euclidean space.\\
\indent{\em 2. Balmus-Montaldo-Oniciuc Conjecture \cite{BMO}}: (i) Any biharmonic submanifold in a sphere has constant mean curvature. (ii) Any proper biharmonic hypersurface in $S^{n+1}$ is (a part of) $S^n(\frac{1}{\sqrt{2}})$ or $S^p(\frac{1}{\sqrt{2}})\times S^{n-p}(\frac{1}{\sqrt{2}})\; (p\ne n/2)$.\\

All of the above conjectures are still open. We refer the reader to a recent book \cite{OC} and the vast references therein for a more detailed account on biharmonic maps and biharmonic submanifolds,  including basic examples, properties of biharmonic maps, some recent progress on biharmonic submanifolds, biharmonic conformal maps,  biharmonic maps with symmetry, and Liouville type and unique continuation theorems for biharmonic maps.\\

\section{ Bi-eigenmaps and some classifications of biharmonic submanifolds in spheres}

Let $\Delta$ be the Laplace operator acting on functions on a Riemannian manifold $(M, g)$ with the convention that $\Delta =div (\nabla f)$. 
\begin{definition}
An eigenfunction of the Laplace operator with eigenvalue $\lambda$ is function $f\in C^{\infty}(M)$ that solves the PDE
\begin{align}\label{Egen}
\Delta f=-\lambda f.
\end{align}
A function $f$ is called an eigenfunction of the bi-Laplace operator (or biharmonic operator termed by some authors) with eigenvalue $\mu$ if it solves the PDE
\begin{align}\label{2-E}
\Delta^2 f=\mu f.
\end{align}
For convenience, we will call an eigenfunction of the bi-Laplace operator, i.e., a solution of (\ref{2-E}) a {\bf bi-eigenfunction}  with {\bf bi-eigenvalue} $\mu$.

A function $f$ is called a {\bf buckling eigenfunction} with eigenvalue $\rho$ if it solves the PDE
\begin{align}\label{2-ED}
\Delta^2 f=-\rho \Delta f.
\end{align}
\end{definition}
Eigenvalue problems (\ref{2-E}) and (\ref{2-ED}) with boundary conditions $f|_{\partial M}=\frac{\partial f}{\partial\nu}|_{\partial M}=0$ are called {\bf the clamped plate problem} and {\bf the buckling problem} respectively, see \cite{Xi} and the references therein for more detailed background and some recent development.

For some universal inequalities for bi-eigenvalues  on compact connected domains in a unit $m$-sphere with boundary conditions see \cite{WX1, WX2} and \cite{Xi}.\\

It is clear from the definitions that (i) any eigenfunction of the Laplace operator of eigenvalue $\lambda$ is always a bi-eigenfunction of  bi-eigenvalue $\mu=\lambda^2$, (ii) any eigenfunction of the Laplace operator of eigenvalue $\lambda$ is also a buckling eigenfunction of the buckling eigenvalue $\rho=\lambda$. It is also easily seen that there are bi-eigenfunctions which are not eigenfunctions, e.g., $f(x)=|x|^2$ is a bi-eigenfunction on $\r^{m+1}$ with bi-eigenvalue $\mu=0$ which is not an eigenfunction. For this reason, we will use the term {\em proper bi-eigenfunction} for an bi-eigenfunction which is not an eigenfunction.

It is well known (see e.g., \cite{Xi}) that all of the above three types of eigenvalues are countable so we can list them as

\begin{align}\notag
0= & \lambda_1<\lambda_2\le \lambda_3\le \cdots +\infty\hskip2cm  (\rm Laplacian\; eigenvalues)\\\notag
0= & \mu_1 < \mu_2\le \mu_3\le \cdots +\infty\hskip2cm (\rm bi-eigenvalues)\\\notag
0= & \rho_1 < \rho_2\le \rho_3\le \cdots +\infty\hskip2cm (\rm Buckling\; eigenvalues)\\\notag
\end{align}

The interesting links between eigenfunctions of Laplacian and  minimal submanifolds traces back to the well known work of
Eells-Sampson \cite{ES} which states that an isometric immersion $\phi: M^m\to \r^{n+1}$ is minimal if and only if $\Delta \phi=0$. This means all the component functions are eigenfunctions of zero eigenvalues.

The following theorem is also well known.
\begin{theorem}\label{T0}
 (Takahashi \cite{Ta}) An isometric immersion  $\phi: M^m\to \r^{n+1}$ with $\Delta \phi=-\lambda\, \phi$ for a constant $\lambda\ne 0$, then $\lambda>0$ and $\phi$ realizes as a minimal isometric immersion $\phi: M^m\to S^{n}(\sqrt{\frac{m}{\lambda}})$. Conversely, if an isometric immersion $\phi: M^m\to \r^{n+1}$ realizes and a minimal isometric immersion into a sphere $S^n(r)$ in $\r^{n+1}$, then  $\Delta \phi=-\lambda \phi$ for $\lambda=\frac{m}{r^2}$. \\
\end{theorem}
On the setting of biharmonic submanifolds, it was proved by B. Y. Chen \cite{Ch} that any spherical biharmonic submanifold in Euclidean space is minimal. More precisely, one can actually prove that there exists no biharmonic  isometric immersion $\phi:M^m\to \r^{n+1}$  ( an isometric immersion with $\Delta^2 \phi=0$) satisfying $\phi(M)\subset S^{n}$. For the case of hypersurfaces, it was proved in \cite{Vi} that  a biharmonic hypersurface $\phi: M^n\to S^{n+1}$ with $\Delta^2 \phi=f\phi$ for some function $f$ on $M$ is minimal.

It would be interesting to know whether there is an  isometric immersion $\phi: M^m\to \r^{n+1}$ with $\Delta^2 \phi=\lambda\, \phi$ or  $\Delta^2 \phi=-\mu \Delta\, \phi$ for some constants $\lambda, \mu$ that realizes as a biharmonic submanifold $\phi: M\to S^n$. Interestingly, we do have
\begin{example}
The composition $\phi=i\circ\varphi$ of a minimal  isometric immersion $\varphi: M^m\to S^{n-1}(\frac{1}{\sqrt{2}})$ followed by the standard biharmonic isometric immersion\\ $i: S^{n-1}(\frac{1}{\sqrt{2}})\to S^n\subset \r^{n+1}$ is a buckling eigenmap with $\Delta^2\phi=-2m\Delta \phi$ which realizes as a proper biharmonic submanifold.\\

In fact, we know (see \cite{CMO2}) that $\phi$ is a proper biharmonic submanifold. One can easily check that in this case, $\phi(x)= (\varphi(x), 1/\sqrt{2}\,)$. So, using Theorem \ref{T0}, we have
\begin{align}\notag
\Delta \phi & = (\Delta \varphi, 0)=(-2m\varphi, 0)\ne \lambda \phi,\; {\rm and}\\ \notag
\Delta^2 \phi & =(4m^2\varphi, 0)=-2m\Delta \phi,
\end{align}
which means exactly that the isometric immersion $\phi: M\to S^n\subset\r^{n+1}$  is a buckling eigenmap that realizes as a proper biharmonic submanifold in  the sphere $S^n$.
\end{example}

In this short note, we first give a complete classification of biharmonic submanifolds in a sphere defined by bi-eigenmaps or buckling eigenmaps, which can be viewed as a generalization of Takahashi's characterization of  minimal submanifolds in a sphere  by eigenmaps. We then generalize the results to have a classification of  biharmonic bi-eigenmaps and buckling eigenmaps into spheres with constant energy density. 

First, we prove the following theorem which generalizes the result on case of biharmonic hypersurfaces obtained in \cite{Vi}.
\begin{theorem}\label{T1}
If an isometric immersion $\phi: M^m\to \r^{n+1}$ with $\Delta^2 \phi=\lambda\, \phi$ for some function $\lambda$ on $M$ that realizes as a biharmonic submanifold $\phi: M\to S^n$, then  $\lambda=m^2$, and the biharmonic submanifold $\phi: M\to S^n$ is actually minimal. 
\end{theorem}
\begin{proof}
Note that an isometric immersion $\phi: M^m\to \r^{n+1}$ realizes as a biharmonic submanifold $\phi: M\to S^n$ means that the isometric immersion $\phi: M^m\to \r^{n+1}$ is the composition $\phi: M\to S^n\to \r^{n+1}$. It is easily checked (see e.g., \cite{OC}, Page 358) that
\begin{align}\label{H}
\Delta \phi=m\eta-m \phi,
\end{align}
where $\eta$ is the mean curvature vector field of the submanifold $\phi: M\to S^n$.
Since $\eta$ is tangent to $S^n$ whilst $\phi$ is normal to $S^n$, we have 
\begin{align}\label{Mo}
|\Delta \phi|^2=m^2|\eta|^2+m^2.
\end{align}

On the other hand, It is known \cite{CMO2} (see also Equation (7.12b) in \cite{OC}) that an isometric immersion $M^m\to S^{n}$ is biharmonic if and only if
\begin{align}\label{102}
\Delta^2 \phi=-[2m\Delta \phi+(2m^2-|\Delta \phi|^2)\phi].
\end{align}
Substituting $\Delta^2\phi=\lambda \phi$, (\ref{H}) and (\ref{Mo}) into (\ref{102}) yields
\begin{align}\label{103}
2m^2\eta=-[\lambda-m^2(1+|\eta |^2)]\phi.
\end{align}
Using the fact that $\eta$ is tangent to $S^n$ whilst $\phi$ is normal to $S^n$ we conclude that $\eta=0$ and hence the submanifold is minimal and $\lambda=m^2$. Thus, we obtain the theorem. 
\end{proof}
\begin{remark}
Note that Theorem \ref{T1}  implies in particular that a biharmonic isometric immersion $\phi: M^m\to S^{n}$ defined by a bi-eigenmap $\Delta^2\phi=\lambda \phi$  is an eigenmap $\Delta \phi=-m\phi$.
\end{remark}

\begin{theorem}\label{T2}
If an isometric immersion $\phi: M^m\to \r^{n+1}$ with $\Delta^2 \phi=-\mu\, \Delta \phi$ for some function $\mu$ on $M$ that realizes as a biharmonic submanifold $\phi: M\to S^n$, then  $ \mu =2m$ and the biharmonic submanifold $\phi: M\to S^n$ has constant mean mean curvature $|\eta|=1$, and hence it is the composition of a minimal submanifold $\psi: M^m\to S^{n-1}(\frac{1}{\sqrt{2}})$ followed by the standard biharmonic isometric immersion $S^{n-1}(\frac{1}{\sqrt{2}})\to S^n$. 
\end{theorem}
\begin{proof}
Substituting $\Delta^2\phi=-\mu \Delta \phi$, (\ref{H}) and (\ref{Mo}) into (\ref{102}) yields
\begin{align}\label{108}
0=m(\mu-2m) \eta-[m(\mu-2m)+m^2(1-|\eta|^2)\phi
\end{align}
Since $\eta$ is tangent to $S^n$ whilst $\phi$ is normal to $S^n$, we have $\mu= 2m$ and $|\eta|=1$. Using Oniciuc's result \cite{On} (see also Theorem 7.4 in \cite{OC}) we conclude that the biharmonic submanifold $\phi: M\to S^n$ is the composition of a minimal submanifold $\psi: M^m\to S^{n-1}(\frac{1}{\sqrt{2}})$ followed by the standard biharmonic isometric immersion $S^{n-1}(\frac{1}{\sqrt{2}})\to S^n$.  This completes the proof of the theorem.
\end{proof}
\begin{remark}
Note that Theorem \ref{T2} impies in particular that there exists biharmonic isometric immersion $\phi: M^m\to S^{n}$ defined by buckling eigenmap $\Delta^2\phi=-\mu \Delta \phi$ which is not an eigenmap, i.e.,  $\Delta \phi\ne \lambda \phi$.
\end{remark}

\section{Biharmonic bi-eigenmaps and buckling eigenmaps into a sphere}

Note that an isometric immersion $\varphi: (M^m, g) \longrightarrow S^n$  that defines a submanifold has constant energy density $|d\varphi|^2=m$. In this section, we will generalize Theorems \ref{T1}, \ref{T2} to the cases of maps with constant energy density.

First, we recall that  for a map  $\varphi: (M^m, g) \longrightarrow S^n$ into a sphere and the standard isometric embedding $i: S^{n}\longrightarrow \mathbb{R}^{n+1}$, we have 
 \begin{theorem}\label{O1} {\rm (\cite{OO}, see also Lemma 10.3 in \cite{OC})}\\
 The  map $\varphi$ is  biharmonic if and only if  the map $\phi=i\circ\varphi: S^m\to \r^{n+1}$ solves the following PDE:
\begin{eqnarray}\notag
&&\Delta^2\phi +2|{\rm d} \phi|^2\Delta \phi+(\Delta |{\rm d}\phi|^2+2\,{\rm div}\,\theta-| \Delta \phi |^2+2|{\rm d}\phi|^4 )\phi \\\label{MF}
&&+2\,{\rm d}\phi({\rm grad}  |{\rm d} \phi|^2)=0,
\end{eqnarray}
where  $\theta=\langle {\rm d}\varphi, \tau(\varphi)\rangle$ a  $1$-form on $M$ defined by $\theta (X)=\langle {\rm d}\varphi, \tau(\varphi)\rangle(X)=\langle {\rm d}\varphi(X), \tau(\varphi)\rangle$.
\end{theorem}

\begin{theorem}\label{T3}
A biharmonic map $\varphi: M^m\to S^n$ with constant energy density and  $\Delta^2 \phi=\lambda\, \phi$ for some function $\lambda$ on $M$ is harmonic, where $\phi=i\circ \varphi$ and  $i: S^{n}\longrightarrow \mathbb{R}^{n+1}$ is the standard isometric embedding.
\end{theorem}
\begin{proof}
To prove this and the next theorem, we need the following \\

{\bf Claim I:}  A map $\varphi: M^m\to S^n$ into sphere with constant energy density 
\begin{equation}
|d\varphi|^2=c
\end{equation}
is biharmonic if and only if
\begin{eqnarray}\label{Me1}
\Delta^2\phi +2c\, \Delta \phi+(2c^2-\langle  \Delta^2\phi, \phi\rangle )\phi =0.
\end{eqnarray}
{\bf Proof of  Claim I:} It is easily checked that $|d\phi|^2=|d\varphi|^2=c$. Using this, together with (see Equation (10.89))
\begin{eqnarray} 
div\, \theta= |\Delta \phi|^2+\langle \nabla \Delta \phi, \nabla\phi\rangle,
\end{eqnarray}
we can rewrite Equation (\ref{MF}) as
\begin{eqnarray}\label{107}
\Delta^2\phi +2 c\,\Delta \phi+(2c^2+ |\Delta \phi|^2+2\langle \nabla \Delta \phi, \nabla\phi\rangle )\phi =0.
\end{eqnarray}
On the other hand, since $\langle \phi, \phi\rangle =1$, we have
\begin{align}\notag
0=\Delta |\phi|^2=2\langle \Delta \phi, \phi\rangle+2|d\phi|^2, \;or
\end{align}
\begin{align}\label{106}
\langle \Delta \phi, \phi\rangle=-|d\phi|^2.  
\end{align}
Applying  $\Delta$ on both sides of (\ref{106}) yields
\begin{align}\notag
|\Delta \phi|^2+\langle \Delta^2 \phi, \phi\rangle+2\langle \nabla \Delta \phi, \nabla\phi\rangle=-\Delta |d\phi|^2. 
\end{align}
From this and the assumption that $|d\phi|^2=c$, we have
\begin{align}\notag
|\Delta \phi|^2+2\langle \nabla \Delta \phi, \nabla\phi\rangle=-\langle \Delta^2 \phi, \phi\rangle. 
\end{align}
Substituting this into (\ref{107}) we obtain the claim.

To complete the proof of the theorem, we use the assumption  $\Delta^2 \phi=\lambda\, \phi$ for some function $\lambda$ on $M$ and Equation (\ref{Me1}) to have
\begin{eqnarray}\notag
&& \lambda \phi +2c\, \Delta \phi+(2c^2-\lambda )\phi =0,\;{\rm or}\\\notag
&& 2c\, \Delta \phi+2c^2\phi=0.
\end{eqnarray}
Combining this and the well known tension field formula 
\begin{equation}\label{TF}
\Delta \phi=\tau (\varphi)-|d\phi|^2 \phi
\end{equation}  
we have $2c\tau(\varphi)\equiv 0$. From this we conclude that either $c=|d\varphi|^2=0$ which implies $\varphi$ is a constant map and hence harmonic, or $\tau(\varphi)\equiv 0$ which also means  $\varphi$ is  a harmonic map. Thus, we obtain the theorem.
\end{proof}

\begin{theorem}\label{T4}
A biharmonic map $\varphi: (M^m,g) \to S^n$ into a sphere with constant energy density $|d\varphi|^2=c$ and $\Delta^2 \phi=-\mu\, \Delta \phi$ for some function $\mu$ on $M$ is either harmonic or $ \mu =2c$. 
\end{theorem}
\begin{proof}
Substituting $\Delta^2\phi=-\mu \Delta \phi$ and (\ref{106}) with $|d\phi|^2=c$ into (\ref{Me1}) yields
\begin{align}\label{108}
(2c-\mu) \Delta \phi+c(2c-\mu )\phi =0.
\end{align}
It follows from this and (\ref{TF}) that $ (2c-\mu) \tau(\varphi)=0$. Thus, either  $\mu= 2c$, or $\tau(\varphi) = 0$, in which case $\varphi$ is harmonic. This completes the proof of the theorem.
\end{proof}

\begin{remark}
We would like to point out again when the map $\phi: (M^m,g) \to S^n$ is  an isometric immersion, then $|d\varphi|^2=m=c$, so Theorems \ref{T3}  and \ref{T4}  include Theorems \ref{T1} and \ref{T2} as special cases respectively. Also, Example 1 provides an example of a biharmonic buckling eigenmap into a sphere with constant energy density.
\end{remark}

\begin{ack}
I would like to thank M. Vieira for some valuable comments that lead to some improvements of an earlier version of the paper.
\end{ack}

\end{document}